\documentclass[reqno,12pt]{amsproc}
\usepackage{amsthm,wrapfig,amsmath,amssymb,amsfonts,setspace,verbatim,graphicx,mathtools,mathrsfs,commath,float,mdframed,frame,xcolor,qtree,enumitem, ulem,afterpage,bm}

\newtheorem{ut}{Theorem}
\numberwithin{equation}{section}
\newtheorem{up}[ut]{Proposition}

\theoremstyle{definition}

\newtheorem{uq}{Question}

\usepackage[a4paper,margin=1.1in]{geometry}

\author[D. S. Lipham]{David S. Lipham}

\address{Department of Mathematics, College of Coastal Georgia, Brunswick, GA 31520, United States of America}
\email{dlipham@ccga.edu}

\title[The $\sigma$-product of  complete Erd\H{o}s space]{The $\bm{\sigma}$-product of  complete Erd\H{o}s space}

\begin{document}

\begin{abstract}We show that the $\sigma$-product of complete Erd\H{o}s space $\mathfrak E_{\mathrm{c}}$ is homeomorphic to the rational product $\mathbb Q\times \mathfrak E_{\mathrm{c}}$, answering a  question by Rodrigo Hern\'{a}ndez-Guti\'{e}rrez and Alfredo Zaragoza.
\end{abstract}

\maketitle

 \section{Introduction}
 
Given a topological space $X$ and a point $p\in X$, the \textbf{$\sigma$-product of $X$} is defined to be the following subspace of $X^\omega$: $$\sigma X=\{\mathbf x \in X^\omega:x_n=p\text{ for all but finitely many }n<\omega\}.$$ Assuming that $X$ is homogeneous (e.g.\ a topological group),  this definition is invariant of the choice of  $p$. Another way of writing $\sigma X$ is as the union of all finite powers of $X$ inside of $X^\omega$;
 $$\sigma X=\bigcup_{n<\omega} X^n\times \{p\}^{\omega\setminus (n+1)}.$$ 
 
Here we are interested in spaces $X$ for which $\sigma X$ is homeomorphic the rational product $\mathbb Q\times X$. Examples include  $\mathbb Q$,  the Cantor set $2^\omega$, and the space of irrational numbers $\mathbb P$ (see  \cite{van}). An example of positive dimension  is the \textbf{Erd\H{o}s space} $\mathfrak E$, defined as the set of all points in the Hilbert space $\ell^2$ whose coordinates are all rational. In fact, by the characterization of  Erd\H{o}s space in \cite{erd},  $\sigma \mathfrak E$ and $\mathbb Q\times \mathfrak E$ are each homeomorphic to  $\mathfrak E$ (cf.\ \cite[Corollary 9.4]{erd}).\footnote{In \cite{erd}, the $\sigma$-product is called the \textbf{weak product}.} 

Now consider the \textbf{complete Erd\H{o}s space}   \begin{align*}\textstyle\mathfrak E_{\mathrm{c}}&=\{\mathbf x\in \ell^2:x_n\in \{0\}\cup \{1/n:0<n<\omega\} \text{ for all }n<\omega\}\\
&=\textstyle\{\mathbf x\in \ell^2:x_n\in \{1,\frac{1}{2},\frac{1}{3},\ldots,0\}\text{ for all }n<\omega\}.\end{align*}  Like $\mathfrak E$ it is a $1$-dimensional topological group, and  it has the added quality of being  completely metrizable (note that $\textstyle\mathfrak E_{\mathrm{c}}$ is closed in $\ell^2$).   It is topologically equivalent to  the set of endpoints of the Lelek fan, and also to the set of all points in $\ell^2$ whose coordinates are all irrational \cite{31}. In   \cite[Question 6.1]{rz}, Hern\'{a}ndez-Guti\'{e}rrez and Zaragoza asked whether the $\sigma$-product of $\mathfrak E_{\mathrm{c}}$ is homeomorphic to $\mathbb Q\times \mathfrak E_{\mathrm{c}}$. Here we provide a positive answer  using their  characterization of the latter.
 
 \begin{ut}$\sigma \mathfrak E_{\mathrm{c}} \simeq \mathbb Q\times \mathfrak E_{\mathrm{c}}$.\end{ut}
 
We pose the following question in response. 

\begin{uq}Is $\{\mathbf x\in \mathfrak E_{\mathrm{c}}:x_n=0 \text{ for some } n<\omega\}$ homeomorphic to $\mathbb Q\times \mathfrak E_{\mathrm{c}}$?\end{uq}

\section{Properties of $\mathfrak E_{\mathrm{c}}$}

An intersection of clopen subsets of a space $X$ is called a \textbf{C-set} in $X$. A space $X$ with a  basis of C-set neighborhoods is  called  \textbf{almost zero-dimensional}.   A separable metrizable, zero-dimensional topology $\mathcal W$ on  $X$ \textbf{witnesses the almost zero-dimensionality of $X$}, provided that $\mathcal W$ is coarser than the original topology of $X$, and $X$ has a basis of neighborhoods which are closed in $\mathcal W$.  See \cite[Chapter 2]{erd}.

The  zero-dimensional topology  that  $\mathfrak E_{\mathrm{c}}$ inherits from $\{1,\frac{1}{2},\frac{1}{3},\ldots,0\}^\omega$ witnesses that $\mathfrak E_{\mathrm{c}}$ is almost zero-dimensional. This is due to the proposition below and the fact that convergence in $\ell^2$ is stronger than pointwise convergence.  

\begin{up}Every closed $\varepsilon$-ball in $\mathfrak E_{\mathrm{c}}$ 
  is closed in the topological product $\{1,\frac{1}{2},\frac{1}{3},\ldots,0\}^\omega$.\end{up}

\begin{proof}Let $\mathbf x\in \mathfrak E_{\mathrm{c}}$ and $\varepsilon\in (0,\infty)$ be given. Let $$B=\{\mathbf{y}\in \mathfrak E_{\mathrm{c}}:\|\mathbf y-\mathbf x\|\leq\varepsilon\}$$ be the closed $\varepsilon$-ball around $x$ (with respect to the $\ell^2$-norm). For every integer $N$ the function $\varphi_N:\{1,\frac{1}{2},\frac{1}{3},\ldots,0\}^\omega\to [0,\infty)$ defined by $$\varphi_N(\mathbf y)=\sqrt{\sum_{n=0}^N(y_n-x_n)^2}$$ is easily seen to be continuous. Essentially, $\varphi_N$ is a restriction of the Euclidean metric on $\mathbb R^N$. So  $\varphi^{-1}_N[0,\varepsilon]$ is closed in $\{1,\frac{1}{2},\frac{1}{3},\ldots,0\}^\omega$. Hence $$B=\bigcap_{N=0}^\infty \varphi^{-1}_N[0,\varepsilon]$$ is an intersection  of closed subsets of $\{1,\frac{1}{2},\frac{1}{3},\ldots,0\}^\omega$. \end{proof}
  
The other key property of $\mathfrak E_{\mathrm{c}}$ that we will need, was proved by  Erd\H{o}s in \cite{dims}. 

\begin{up}Every non-empty clopen subset of $\mathfrak E_{\mathrm{c}}$ is unbounded in the $\ell^2$-norm.\end{up}

\section{The characterization of  $\mathbb Q\times \mathfrak E_{\mathrm{c}}$}

Following \cite{rz}, we say that a space without  isolated points is \textbf{crowded}.  

An almost zero-dimensional space $E$ belongs to the class $\sigma \mathcal E$ if there exists a witness topology $\mathcal W$ for $E$, a collection $\{E_n:n<\omega\}$ of subsets of $E$, and a basis $\mathcal  O$ of neighborhoods in $E$ such that:

\begin{enumerate}
\item $E=\bigcup \{E_n:n<\omega\}$;
\item each $E_n$ is closed in $(E,\mathcal W)$;
\item each $E_n$ is a crowded nowhere dense subset of $E_{n+1}$ ;
\item each point of $E$ has a neighborhood which contains no non-empty clopen subset of any $E_n$;
\item for every $V\in \mathcal O$ and $n<\omega$,  the set $V\cap E_n$ is compact  in $\mathcal W$.

\end{enumerate}

The main result of \cite{rz} states that $\mathbb Q\times \mathfrak E_{\mathrm{c}}$ is the only member of $\sigma \mathcal E$ (up to homeomorphism). 

\begin{up}[{\cite[Theorem 3.5]{rz}}]$\sigma \mathcal E\approx\{\mathbb Q\times \mathfrak E_{\mathrm{c}}\}$.\end{up}

\section{The proof of Theorem 1}

\noindent We will show  that $\sigma \mathfrak E_{\mathrm{c}}$ belongs to the class $\sigma \mathcal E$.

Endow $\{1,\frac{1}{2},\frac{1}{3},\ldots,0\}^\omega$ with the (zero-dimensional) product topology. Let $\mathcal W$ be the  subspace topology on the set $\mathfrak E_{\mathrm{c}}\subset \{1,\frac{1}{2},\frac{1}{3},\ldots,0\}^\omega$. Per  Section 2, $\mathcal W$ witnesses the almost zero-dimensionality of $\mathfrak E_{\mathrm{c}}$. Let $\sigma \mathcal W$ be  the natural topology on $\sigma \mathfrak E_{\mathrm{c}}$ that is generated by $\mathcal W$; equivalently $\sigma \mathcal W$ is the topology that $\sigma \mathfrak E_{\mathrm{c}}$ inherits from  $\big(\textstyle{\{1,\frac{1}{2},\frac{1}{3},\ldots,0\}^\omega}\big)^\omega.$  It is easily checked that  $\sigma \mathcal W$ witnesses the almost zero-dimensionality  of $\sigma \mathfrak E_{\mathrm{c}}$.
 
We will now assume that the point $p$ used to define $\sigma \mathfrak E_{\mathrm{c}}$ is the point $\mathbf 0\in \mathfrak E_{\mathrm{c}}$.   For every $n<\omega$ let  $K_n=\{\mathbf x\in\mathfrak E_{\mathrm{c}}:\|\mathbf x\|\leq n+1\}$, and  define $$E_n=\mathfrak E_{\mathrm{c}} \times \underbrace{K_{n}\times \ldots \times K_n}_\text{$n$ times}\times \{\mathbf 0\}^{\omega\setminus(n+1)}.$$ Clearly, $\sigma \mathfrak E_{\mathrm{c}}=\bigcup \{E_n:n<\omega\}.$ By Proposition 2, $K_n$ is closed in $\mathcal W$. Therefore $E_n$ is closed in $\sigma\mathcal W$.
Observe also that $E_n$ is a crowded nowhere dense subset of $E_{n+1}$.

Thus far we have established (1) through (3).   Toward proving (4), let $\mathbf y\in \sigma \mathfrak E_{\mathrm{c}}$. Let $U$ be a bounded open set in $ \mathfrak E_{\mathrm{c}}$ with $y_0\in U$.    Then $\{\mathbf z\in\sigma \mathfrak E_{\mathrm{c}}: z_0\in U\}$ is a neighborhood of  $\mathbf y$ in $\sigma \mathfrak E_{\mathrm{c}}$ . Arguing as in \cite[Remark 5.2]{erd}, we show that  $\{\mathbf z\in\sigma \mathfrak E_{\mathrm{c}}: z_0\in U\}$  contains no non-empty clopen subset of any $E_n$. First note that $U$ does not contain any non-empty clopen subset of $\mathfrak E_{\mathrm{c}}$ by Proposition 3. Now  suppose for a contradiction that $C$ is  non-empty and clopen in $E_n$ and $\pi_0[C]\subset U$. Let $\mathbf{z}\in C$. Then $$C'=C\cap \Big(\mathfrak E_{\mathrm{c}}\times \prod_{n\geq 1}\{z_n\}\Big)$$ is  non-empty and clopen in the space $\mathfrak E_{\mathrm{c}}\times \{z_1\}\times \{z_2\}\times \ldots$ where the  projection map $\pi_0$ is both open and closed.  But then  $\pi_0[C']\subset U$ is a non-empty  clopen subset of $\mathfrak E_{\mathrm{c}}$,  a contradiction. Thus (4) holds.

Finally, for (5) let $\beta$ be a  basis  for $\mathfrak E_{\mathrm{c}}$  consisting of neighborhoods which are compact in $\mathcal W$, as provided by Proposition 2.   Let $\mathcal O$ be the collection of  all sets $$V=(B_0\times \ldots \times B_m\times \mathfrak E_{\mathrm{c}}^{\omega\setminus(m+1)})\cap \sigma \mathfrak E_{\mathrm{c}} $$
 where $m\geq 0$ and $B_i\in \mathscr \beta$ for each $i\leq m$. Clearly $\mathcal O$ is  a neighborhood basis for $\sigma \mathfrak E_{\mathrm{c}}$.  For every $V\in \mathcal O$ and $n<\omega$, the intersection $V\cap E_n$ is a product with factors of the kind  $B_0$, $B_i\cap K_n$, $K_n$,  $\{\mathbf 0\}$ or $\varnothing$. Recall from Proposition 2 that each $K_n$ is compact in $\mathcal W$. Thus  $V\cap E_n$ is a  product of  $\mathcal W$-compact sets. So $V\cap E_n$ is compact in $\sigma\mathcal W$. 

We conclude that $\sigma \mathfrak E_{\mathrm{c}}\in \sigma \mathcal E$. By Proposition 4, $\sigma \mathfrak E_{\mathrm{c}} \simeq \mathbb Q\times \mathfrak E_{\mathrm{c}}$.  \hfill$\blacksquare$


\begin{thebibliography}{HD}



\bibitem{erd}J. J. Dijkstra and J. van Mill, Erd\H{o}s space and homeomorphism groups of manifolds, Mem. Amer. Math. Soc. 208 (2010), no. 979.

\bibitem{dims}P. Erd\H{o}s, The dimension of the rational points in Hilbert space, Ann. of Math. (2) 41 (1940), 734--736.

\bibitem{rz}R. Hern\'{a}ndez-Guti\'{e}rrez and A. Zaragoza, A characterization of the product of the rational numbers and complete Erd\H{o}s space, Can. Math. Bull. 66, No. 1, 87--102.

\bibitem{31}K. Kawamura, L. G. Oversteegen, and E. D. Tymchatyn, On homogeneous totally disconnected $1$-dimensional spaces, Fund. Math. 150 (1996), 97--112.

\bibitem{van} J. van Mill, Characterization of some zero-dimensional separable metric spaces, Trans. Amer. Math. Soc. 264 (1981), 205--215.

\end{thebibliography}
\end{document}